\documentclass[10pt]{article}
\usepackage{amssymb,amsmath,amsthm,latexsym}

\newtheorem{theorem}{Theorem}[section]
\newtheorem{lemma}[theorem]{Lemma}

\newtheorem{remark}[theorem]{Remark}
\newtheorem{cor}[theorem]{Corollary}
\newtheorem{prop}[theorem]{Proposition}
\newtheorem{defin}[theorem]{Definition}

\def\cD{\mathcal D}

\def\cA{\mathcal A}
\def\cB{\mathcal B}
\def\cC{\mathcal C}
\def\cE{\mathcal E}
\def\cH{\mathcal H}
\def\cL{\mathcal L}

\def\cO{\mathcal O}
\def\cP{\mathcal P}
\def\cQ{\mathcal Q}

\def\cS{\mathcal S}
\def\cW{\mathcal W}
\def\cR{\mathcal R}
\def\cX{\mathcal X}
\def\GF{{\rm GF}}
\def\PG{{\rm PG}}
\begin{document}


\title{Relative $m$--ovoids of elliptic quadrics}
\author{Antonio Cossidente\\ Dipartimento di Matematica Informatica ed Economia\\ Universit\`a della Basilicata\\
 Contrada Macchia Romana\\ I-85100 Potenza\\ Italy\\antonio.cossidente@unibas.it\\
\\
 Francesco Pavese\\Dipartimento di Meccanica Matematica e Management\\ Politecnico di Bari\\
 Via Orabona, 4\\ I-70125 Bari\\ Italy\\francesco.pavese@poliba.it}
\date{}

\maketitle
\begin{abstract}
Let $\cQ^-(2n+1,q)$ be an elliptic quadric of $\PG(2n+1,q)$. A {\em relative $m$--ovoid} of ${\cal Q}^-(2n+1,q)$ (with respect to a parablic section $\cQ := \cQ(2n,q) \subset {\cal Q}^-(2n+1,q)$) is a subset $\cal R$ of points of ${\cal Q}^-(2n+1,q)\setminus {\cal Q}$ such that every generator of ${\cal Q}^-(2n+1,q)$ not contained in $\cal Q$ meets $\cal R$ in precisely $m$ points. A relative $m$--ovoid having the same size as its complement (in $\cQ^-(2n+1,q) \setminus \cQ$) is called a {\em relative hemisystem}. We show that a nontrivial relative $m$--ovoid of ${\cal Q}^-(2n+1,q)$ is necessarily a relative hemisystem, forcing $q$ to be even. Also, we construct an infinite family of relative hemisystems of ${\cal Q}^-(4n+1,q)$, $n \ge 2$, admitting ${\rm PSp}(2n,q^2)$ as an automorphism group. Finally, some applications are given.
\end{abstract}
\date{}

\section{Introduction}

In 2011 Penttila and Williford \cite{PW} introduced the notion of {\em relative hemisystem} of a Hermitian surface 
${\cal H}(3,q^2)$, $q$ even. A relative hemisystem of ${\cal H}(3,q^2)$ with respect to a symplectic subgeometry ${\cal W}(3,q)\subset {\cal H}(3,q^2)$ is a set $\cal S$ of generators of $\cH(3,q^2)$ disjoint from ${\cal W}(3,q)$ such that each point of ${\cal H}(3,q^2)\setminus {\cal W}(3,q)$ lies on $q/2$ generators of $\cal S$. Relative hemisystems are interesting in view of their connection with primitive $Q$--polynomial association schemes not arising from distance regular graphs \cite[Theorem 3]{PW}. Up to date, three infinite families of relative hemisystems of ${\cal H}(3,q^2)$ and a sporadic example are known \cite{PW}, \cite{AC1}, \cite{AC2}, \cite{CPa}.

In a recent paper Bamberg and Lee \cite{BL} generalized the concept of a relative hemisystem to a {\em relative $m$--cover} of a generalized quadrangle. Let $\Gamma$ be a generalized quadrangle with a subquadrangle $\Gamma'$. A relative $m$--cover of $\Gamma$ is a set $\cal S$ of external lines with respect to $\Gamma'$ such that each point of $\Gamma\setminus\Gamma'$ 
lies on $m$ members of $\cal S$. Bamberg and Lee showed that a relative $m$--cover of a generalized quadrangle of order $(q^2,q)$ with respect to a subquadrangle of order $(q,q)$ must be a hemisystem. In particular, they were able to prove an analogue of Segre's result \cite{SE}, \cite[Section 19.3]{H}. See also \cite{CP}. Indeed they showed that a nontrivial relative $m$--cover of ${\cal H}(3,q^2)$, with respect to a symplectic subgeometry ${\cal W}(3,q)$ is a relative hemisystem, forcing $q$ to be even \cite[Theorem 1.1, Corollary 1.2]{BL}.

An $m$--ovoid of a finite polar space $\cP$ is a set of points which meets every generator of $\cP$ in precisely $m$ points \cite{BKLP}.

Dual of the generalized quadrangles ${\cal H}(3,q^2)$ and ${\cal W}(3,q)$ are the elliptic quadric ${\cal Q}^-(5,q)$ and the parabolic quadric ${\cal Q}(4,q)$, respectively. A relative $m$--cover of ${\cal H}(3,q^2)$ corresponds then to a set of points $\cR$ of ${\cal Q}^-(5,q)\setminus {\cal Q}(4,q)$ such that each generator of ${\cal Q}^-(5,q)$ not on ${\cal Q}(4,q)$ meets $\cal R$ in $m$ points. In this way, the dual of a relative $m$--cover of $\cH(3,q^2)$ can be considered an ``affine'' $m$--ovoid of ${\cal Q}^-(5,q)$.

In this paper we introduce the concept of relative $m$--ovoid of ${\cal Q}^-(2n+1,q)$ with respect to a parabolic section $\cQ := \cQ(2n,q) \subset {\cal Q}^-(2n+1,q)$ as a generalization of the dual of a relative $m$--cover of $\cH(3,q^2)$. A relative $m$--ovoid of ${\cal Q}^-(2n+1,q)$ (with respect to ${\cal Q}$) is a subset $\cal R$ of points of ${\cal Q}^-(2n+1,q)\setminus {\cal Q}$ such that every generator of ${\cal Q}^-(2n+1,q)$ not contained in $\cal Q$ meets $\cal R$ in precisely $m$ points. A relative $m$--ovoid having the same size as its complement (in $\cQ^-(2n+1,q) \setminus \cQ$) is called a {\em relative hemisystem}. 

The paper is organized as follows. In Section \ref{relative} we show that a nontrivial relative $m$--ovoid of ${\cal Q}^-(2n+1,q)$ is necessarily a relative hemisystem, forcing $q$ to be even. This generalizes the main result of \cite{BL}. In Section \ref{family} we construct an infinite family of relative hemisystems of ${\cal Q}^-(4n+1,q)$, $n \ge 2$, admitting ${\rm PSp}(2n,q^2)$ as an automorphism group. In Section \ref{appl} we show that a relative hemisystem of $\cQ^-(2n+1,q)$ gives rise to non--classical $(q^n-1)/(q-1)$--ovoids of the symplectic polar space $\cW(2n+1,q)$, generalizing \cite[Proposition 3.1, Corollary 3.2]{CPa} and, assuming $q=2$ and $n\ge 3$, to a strongly regular graph with parameters $v=2^{n-1}(2^n-1)$, $k=(2^{n-2}-1)(2^n+1)$, $\lambda = 2(2^{n-1}+1)(2^{n-3}-1)$, $\mu = 2^{n-1}(2^{n-2}-1)$.

Throughout the paper we will use the term $n$--space to refer to an $n$--dimensional projective space.
 
\section{Relative $m$--ovoids of ${\cal Q}^-(2n+1,q)$}\label{relative}

Let ${\rm PG}(2n+1,q)$ be the projective space over the finite field ${\rm GF}(q)$ and let ${\cal Q}^-(2n+1,q)$ be an elliptic quadric of  ${\rm PG}(2n+1,q)$. Let ${\cal Q}={\cal Q}(2n,q)$ be a parabolic quadric embedded in ${\cal Q}^-(2n+1,q)$. 
Let $\perp$ denote the polarity of ${\rm PG}(2n+1,q)$ induced by ${\cal Q}^-(2n+1,q)$.
A {\em generator} of ${\cal Q}^-(2n+1,q)$ is a totally singular $(n-1)$--space. A generator of ${\cal Q}^-(2n+1,q)$ either is contained in $\cal Q$ or meets $\cal Q$ in an $(n-2)$--space. Let $\cal B$ denote the set of generators of ${\cal Q}^-(2n+1,q)$ not contained in $\cal Q$. Then $\vert {\cal B}\vert = q(q^n-1)\prod_{i=2}^n (q^i+1)$. 
	\begin{defin}
A relative $m$--ovoid of ${\cal Q}^-(2n+1,q)$ (with respect to ${\cal Q}$) is a subset $\cal R$ of points of ${\cal Q}^-(2n+1,q)\setminus {\cal Q}$ such that every generator of ${\cal Q}^-(2n+1,q)$ not contained in $\cal Q$ meets $\cal R$ in $m$ points. 	
	\end{defin}
\par\noindent
A relative $m$--ovoid is said to be {\em nontrivial} or {\em proper} if $m\ne 0$ or $m\ne q^{n-1}$. A relative $m$--ovoid having the same size as its complement (in $\cQ^-(2n+1,q) \setminus \cQ$) is called a {\em relative hemisystem}. 
	
	\begin{lemma}
A relative $m$--ovoid $\cal R$ of ${\cal Q}^-(2n+1,q)$ has size $mq(q^n-1)$.
	\end{lemma}
	\begin{proof}
We count in two ways the pairs $(P,b)$, where $P\in \cal R$, $b\in {\cal B}$ and $P\in b$: 
$$ \vert {\cal R} \vert \prod_{i=2}^n(q^i+1)=m q(q^n-1)\prod_{i=2}^n (q^i+1)$$ 
and hence $\vert {\cal R} \vert = mq(q^n-1)$.	
	\end{proof}

	\begin{lemma}\label{pre}
Let ${\cal Q}'$ be a parabolic quadric ${\cal Q}(2n,q)$ embedded in ${\cal Q}^-(2n+1,q)$ and let $P$ be a point of $\cQ^-(2n+1,q)$.
	\begin{enumerate}
	\item[a)] If ${\cal Q}\cap {\cal Q}'$ is an elliptic quadric ${\cal Q}^-(2n-1,q)$, then $\vert \cR \cap \cQ' \vert = m(q^n+1)$.
	\item[b)] If ${\cal Q}\cap {\cal Q}'$ is a hyperbolic quadric ${\cal Q}^+(2n-1,q)$, then $\vert \cR \cap \cQ' \vert = m(q^n-1)$.
	\item[c)] If ${\cal Q}\cap {\cal Q}'$ is a cone, then $\vert \cR \cap \cQ' \vert = m q^n$.
	\item[d)] If $P \in \cR$, then $\vert \cR \cap P^\perp \vert = m(q^n+1)-q^n$.
	\item[e)] If $P \notin \cQ \cup \cR$, then $|\cR \cap P^\perp| = m(q^n+1)$.
	\item[f)] If $P \in \cQ$, then $|\cR \cap P^\perp| = m (q^n-q)$.
	\end{enumerate}
	\end{lemma}
	\begin{proof}
In the first three cases, we count in two ways the pairs $(P,b)$, where $P\in {\cal R}\cap {\cal Q}'$, $b \in \cB$ is a generator of ${\cal Q}'$ and $P\in b$:	

$a)$
$$ \vert {\cal R}\cap {\cal Q}'\vert \prod_{i=1}^{n-1}(q^i+1) = m\prod_{i=1}^n (q^i+1)$$ 
and hence $\vert {\cal R}\cap {\cal Q}'\vert = m(q^n+1)$.

$b)$ 
$$ \vert {\cal R}\cap {\cal Q}'\vert \prod_{i=1}^{n-1}(q^i+1) = m \left( \prod_{i=1}^n (q^i+1) - 2 \prod_{i=1}^{n-1} (q^i+1) \right)$$ 
and hence $\vert {\cal R}\cap {\cal Q}'\vert = m(q^n-1)$.

$c)$
$$ \vert {\cal R}\cap {\cal Q}'\vert \prod_{i=1}^{n-1}(q^i+1) = m \left( \prod_{i=1}^n (q^i+1) - \prod_{i=1}^{n-1} (q^i+1) \right)$$ 
and hence $\vert {\cal R}\cap {\cal Q}'\vert = m q^n$.

$d)$
We count in two ways the pairs $(T,b)$, where $T\in {\cal R}\cap (P^\perp\setminus\{P\})$, $b\in\cal B$, $P\in b$ and $T\in b$:
$$ \vert {\cal R}\cap (P^\perp\setminus\{P\})\vert \prod_{i=2}^{n-1}(q^i+1) = (m-1) \prod_{i=2}^n (q^i+1)$$ 
and hence $\vert {\cal R}\cap P^\perp\vert = m(q^n+1)-q^n$.

In the remaining two cases, we count in two ways the pairs $(T,b)$, where $T\in {\cal R}\cap P^\perp$, $b\in\cal B$, $P\in b$ and $T\in b$:

$e)$
$$ \vert {\cal R}\cap P^\perp\vert \prod_{i=2}^{n-1}(q^i+1) = m \prod_{i=2}^n (q^i+1)$$ 
and hence $\vert {\cal R}\cap P^\perp\vert = m(q^n+1)$.	

$f)$
$$ \vert {\cal R}\cap P^\perp\vert \prod_{i=2}^{n-1}(q^i+1) = m \left( \prod_{i=2}^n (q^i+1) - \prod_{i=1}^{n-1} (q^i+1) \right)$$ 
and hence $\vert {\cal R}\cap P^\perp\vert = m (q^n - q)$.	
	\end{proof}

	\begin{theorem}\label{main}
Let $\cal R$ be a proper relative $m$--ovoid of ${\cal Q}^-(2n+1,q)$. Then $q$ is even and $\cR$ is a relative hemisystem.
	\end{theorem}	
	\begin{proof}
Let ${\cal Q}^-(2n-1,q)$ be an elliptic quadric contained in $\cal Q$. Let $\cal F$ be the pencil of quadrics embedded in ${\cal Q}^-(2n+1,q)$ having as base locus ${\cal Q}^-(2n-1,q)$. Then $\cal F$ contains, apart from $\cal Q$, $q-2$ parabolic quadrics and two cones.
Let $P_1,P_2$ be the vertices of the two cones. We distinguish three cases according as $P_1,P_2\in \cal R$, $P_1,P_2\not\in \cal R$, $P_1\in {\cal R}, P_2\not\in {\cal R}$, respectively.
If the first case occurs then, from Lemma \ref{pre} $a)$, $d)$, we have 
$$(q-2)m(q^n+1)+2m(q^n+1)-2q^n=mq(q^n-1)$$ and hence $m=q^{n-1}$. If the second case occurs then, from Lemma \ref{pre} $a)$, $e)$, we have $$(q-2)m(q^n+1)+2m(q^n+1)=mq(q^n-1)$$ and hence $m=0$. If the third case occurs then, from Lemma \ref{pre} $a)$, $d)$, $e)$, we have $$(q-2)m(q^n+1)+2m(q^n+1)-q^n=mq(q^n-1)$$ and hence $2m=q^{n-1}$. It follows that $q$ is even and $m = q^{n-1}/2$. 
	\end{proof}
	
\begin{remark}
{\rm Theorem \ref{main} does not hold true for hyperbolic quadrics as shown in the following example. Let ${\cal Q}^-(3,q)\subset \cQ := {\cal Q}(4,q) \subset {\cal Q}^+(5,q)$, $q$ even. There exists a pencil $\cal P$ of parabolic quadrics having ${\cal Q}^-(3,q)$ as base locus. A plane of ${\cal Q}^+(5,q)$ meets a parabolic quadric of $\cal P$ in a line having necessarily a point on ${\cal Q}^-(3,q)$ and hence meets each of the $q$ sets ${\cal Q}(4,q)\setminus{\cal Q}^-(3,q)$ in $q$ points, i.e., each of them is a ``relative $q$--ovoid'' of $\cQ^+(5,q)$ with respect to $\cQ$. Of course, gluing together $i$ of such sets, a ``relative $i q$--ovoid'' is obtained, $1 \le i \le q$. Hence there exist ``relative $m$--ovoids'' of $\cQ^+(5,q)$ with respect to $\cQ$ that are not relative hemisystems.}
\end{remark}

In view of Theorem \ref{main}, we will assume for the remaining part of the paper that $q$ is even.

Let $G$ denote the stabilizer of ${\cal Q}^-(2n+1,q)$ in ${\rm PGL}(2n+2,q)$.  Let $\Sigma$ be the hyperplane of ${\rm PG}(2n+1,q)$ containing $\cal Q$ and let $N=\Sigma^\perp$. Note that $N$ is the nucleus of $\cal Q$. Let $\tau$ be the involution of $G$ fixing $\Sigma$ pointwise and $N$ linewise.	
		
	\begin{cor}\label{inv}
If $\cal R$ is a relative hemisystem of ${\cal Q}^-(2n+1,q)$, then $| {\cal R} \cap {\cal R}^\tau | = 0$. 		
	\end{cor}
	\begin{proof}
With the same notation introduced in the proof of Theorem \ref{main} we only observe that the line $P_1P_2$ contains $N$. 	
	\end{proof}
	
	\begin{remark}
{\rm 	Note that Theorem \ref{main} and Corollary \ref{inv} generalize \cite[Theorem 1.1 $(a), (b)$]{BL} when the generalized quadrangle is classical.}
	\end{remark}	
	
	\section{Relative hemisystems of ${\cal Q}^-(4n+1,q)$}\label{family}
	
In this Section we construct a relative hemisystem of ${\cal Q}^-(4n+1,q)$ admitting the group ${\rm PSp}(2n,q^2)$ as an automorphism group.

Let $V$ be a $2n$--dimensional vector space over $\GF(q^2)$ and let $k$ be an alternating form with associated polar space ${\cal W}(2n-1,q^2)$. The vector space $V$ can be considered as a $4n$--dimensional vector space $V'$ over $\GF(q)$ via the inclusion $\GF(q) \subset \GF(q^2)$. Composition of $k$ with the trace map $T : z \in \GF(q^2) \mapsto z+z^q \in \GF(q)$ provides a new alternating form $k'$ on $V'$ and so we obtain a new polar space ${\cal W}(4n-1,q)$. Now each point in ${\rm PG}(2n-1,q^2)$ corresponds to a $1$--dimensional vector space in $V$, which in turn corresponds to a $2$--dimensional vector space in $V'$, that is a $1$--dimensional projective space of ${\rm PG}(4n-1,q)$. Extending this map from points of ${\rm PG}(2n-1,q^2)$ to subspaces of ${\rm PG}(2n-1,q^2)$, we obtain an injective map from subspaces of ${\PG}(2n-1,q^2)$ to certain subspaces of $\PG(4n-1,q)$:
$$
 \phi: \PG(2n-1,q^2) \rightarrow \PG(4n-1,q).
$$
The map $\phi$ is called the $\GF(q)$--{\em linear representation} of $\PG(2n-1,q^2)$.

A {\em line spread} of a polar space ${\cP}$ is a collection $\cS$ of mutually disjoint lines of ${\cP}$ such that each point of $\cP$ is contained in an element of $\cS$. The set $\cD=\{\phi(P) \; | \;\; P \in {\cal W}(2n-1,q^2)\}$ is an example of line spread of ${\cal W}(4n-1,q)$. The incidence structure whose points are the elements of $\cD$ and whose lines are the generators of ${\cal W}(4n-1,q)$ spanned by two distinct elements of $\cD$, is isomorphic to ${\cal W}(2n-1,q^2)$. Denote by $\rho$ the symplectic polarity of $\PG(2n-1,q^2)$ associated with ${\cal W}(2n-1,q^2)$ and by $\rho'$ the symplectic polarity of $\PG(4n-1,q)$ associated with ${\cal W}(4n-1,q)$

Since $q$ is even, the symplectic polar space ${\cal W}(4n-1,q)$ is isomorphic to the parabolic quadric ${\cal Q} = {\cal Q}(4n,q)$. This isomorphism can be realized by projection of ${\cal Q}$ from its nucleus $N$ onto a hyperplane not containing $N$ \cite{TJ}. Denote by $\mu$ the bijection between ${\cal W}(4n-1,q)$ and ${\cal Q}$. Thus, $\mu(\cD)$  is a line spread of $\cal Q$ admitting ${\rm PSp}(2n,q^2)$ as an automorphism group.  

Let $\Sigma$ be the ambient projective space of $\cal Q$ and consider $\Sigma$ as a hyperplane of a $\PG(4n+1,q)$. Let ${\cal Q}^-(4n+1,q)$ be an elliptic quadric of $\PG(4n+1,q)$ such that ${\cal Q}^-(4n+1,q) \cap \Sigma = {\cal Q}$. In this way, if $G$ denotes the stabilizer of ${\cal Q}^-(4n+1,q)$ in ${\rm PGL}(4n+2,q)$, by Witt's theorem, we can say that there exists a subgroup $K$ of $G$, isomorphic to ${\rm PSp}(2n,q^2)$, that stabilizes the line spread $\mu(\cD)$ of $\cal Q$. 
   
Let $\perp$ denote the orthogonal polarity of $\PG(4n+1,q)$ associated with ${\cal Q}^-(4n+1,q)$. Then $N=\Sigma^\perp$. Let $\tau$ be the involution of $G$ fixing $\Sigma$ pointwise and $N$ linewise. We are interested in the action of $K$ on points of ${\cal Q}^-(4n+1,q) \setminus {\cal Q}$.

\begin{prop}\label{points}
The group $K$ has two orbits $\cO_1, \cO_2$ of the same size on points of ${\cal Q}^-(4n+1,q) \setminus {\cal Q}$. Moreover, $\cO_1^\tau = \cO_2$.   
\end{prop}
\begin{proof}
Let $P_1 \in \cQ^-(4n+1,q) \setminus \cQ$, then $P_1^\perp$ meets $\cQ$ in an elliptic quadric $\cQ^-(4n-1,q)$. An easy counting argument shows that $\cQ^-(4n-1,q)$ contains exactly $|\cQ^-(2n-1,q^2)|$ lines of $\mu(\cD)$. If $\cal E$ is an elliptic quadric of $\PG(2n-1,q^2)$ inducing the polarity $\rho$, then $\phi(\cE)$ defines a set consisting of $|\cE|$ lines of $\cD$ that are contained in an elliptic quadric $\cE'$ of $\PG(4n-1,q)$ that in turn induces the polarity $\rho'$. Note that $\mu(\cE')$ is an elliptic quadric $\cQ^-(4n-1,q)$ contained in $\cQ$. Moreover, the lines of $\mu(\cD)$ contained in $\mu(\cE')$ span a $\PG(4n-1,q)$. On the other hand the number of elliptic quadrics of $\PG(2n-1,q^2)$ inducing the polarity $\rho$ equals the number of elliptic quadrics of $\PG(4n-1,q)$ contained in $\cQ$, indeed
$$
|{\rm PSp}(2n,q^2) : {\rm PGO}^-(2n,q^2)| = |{\rm PSp}(4n,q) : {\rm PGO}^-(4n,q)| = \frac{q^{2n}(q^{2n}-1)}{2} .
$$
It follows that there exists a one-to-one correspondence between the elliptic quadrics of $\PG(2n-1,q^2)$ inducing the polarity $\rho$ and the elliptic quadrics of $\PG(4n-1,q)$ contained in $\cQ$. Since ${\rm PGO}^-(2n,q^2)$ is maximal in $K$ \cite{RHD}, we have that the stabilizer of $P_1$ in $K$ is isomorphic to ${\rm PGO}^-(2n,q^2)$ and hence from the Orbit--Stabilizer theorem $|P_1^K| = q^{2n}(q^{2n}-1)/2$. 

Let $\ell$ be the line joining $P_1$ and $N$ and let $\ell \cap \cQ^-(4n+1,q) = \{P_1, P_2\}$. Since $K$ fixes $N$, if there was an element $\alpha \in K$ such that $P_1^\alpha = P_2$, then $\alpha$ would fix $\ell$ and hence $\alpha$ would lie in the stabilizer of $\ell$ in $K$. On the other hand, since $P_1^\perp \cap \Sigma = \ell^\perp$, the stabilizer of $\ell$ in $K$ coincides with the stabilizer of $P_1$ in $K$, which acts identically on $\ell$. It follows that $P_2 \notin P_1^K$. In particular $K$ has two orbits $\cO_1$ and $\cO_2$ on points of $\cQ^-(4n+1,q) \setminus \cQ$, both of size $q^{2n}(q^{2n}-1)/2$, 
 and if $\ell$ is a secant line to $\cQ^-(4n-1,q)$ containing $N$, then $|\ell \cap \cO_1| = |\ell \cap \cO_2| = 1$.     
\end{proof}

From now on we assume that $n\ge2$. We are now interested in determining the orbits of the group $K$ on the set $\cal L$ consisting of lines of $\cQ^-(4n+1,q)$ not contained in $\cQ$.  

\begin{prop}\label{lines}
The group $K$ has three orbits on lines of $\cal L$.
\end{prop}
\begin{proof}
With the same notation introduced in Proposition \ref{points}, let $P_1 \in \cO_1$ and let $P_2 = P_1^\tau$. Through $P_1$ there pass $|\cQ^-(4n-1,q)|$ lines of $\cal L$ and, as already observed, $\Gamma = P_1^\perp \cap \Sigma$ meets $\cQ^-(4n+1,q)$ in a $\cQ^-(4n-1,q)$ that corresponds to an elliptic quadric $\cQ^-(2n-1,q^2)$ inducing the polarity $\rho$. Note that $N\not\in\Gamma$. Hence $\cQ^-(4n-1,q)$ contains $|\cQ^-(2n-1,q^2)|$ lines of $\mu(\cD)$. Let $X$ denote the set of points covered by the lines of $\mu(\cD)$ that are contained in $\cQ^-(4n-1,q)$. 

Let $\ell$ be a line of $\cQ^-(4n+1,q)$ passing through $P_1$ and meeting $\cQ$ at the point $P$. Let $\ell_P$ be the unique line of $\mu(\cD)$ through $P$ and let $\ell' = P_2 P$. The plane $\pi = \langle \ell, \ell_P \rangle$ is either contained in $\cQ^-(4n+1,q)$ or shares with $\cQ^-(4n+1,q)$ the $2q+1$ points of $\ell \cup \ell_P$. Let $\pi' = \langle \ell_P, P_2 \rangle$

If $P \in X$, then $\pi$ is contained in $\cQ^-(4n+1,q)$ and the solid $\langle \pi, P_2 \rangle$ meets $\cQ^-(4n+1,q)$ in the union of the two planes, $\pi$ and $\pi'$. Let $L = \phi^{-1}(\mu^{-1}(\ell_P))$. Then $L \in \cQ^-(2n-1,q^2)$. If $r$ is a secant line to $\cQ^-(2n-1,q^2)$ through $L$, then $\mu(\phi( r ))$ is a set of $q^2+1$ lines of $\mu(\cD)$ partitioning the point set of a $\cQ(4,q)$ contained in $\cQ$. In particular $\ell_P \in \mu(\phi( r ))$ and $\cQ(4,q)$ meets $\cQ^-(4n-1,q)$ in a hyperbolic quadric $\cQ^+(3,q)$. Indeed, since a line of $\mu(\cD)$ is either contained in $\cQ^-(4n-1,q)$ or it is tangent to it and $\cQ(4,q)$ contains $q^2+1$ lines of $\mu(\cD)$, it follows that $|\cQ(4,q)\cap\cQ^-(4n-1,q)|=q^2-1+2(q+1)=(q+1)^2$. Let $\cE = \cQ^-(4n-5,q)$ be the elliptic quadric $\langle \cQ^+(3,q) \rangle^{\perp_{|\Gamma}}$ and let $\cQ^+(5,q) = \langle \cE \rangle^\perp \cap \cQ^-(4n+1,q)$. By construction $\pi$ and $\pi'$ are planes of $\cQ^+(5,q)$ and $\cQ(4,q) \subset \cQ^+(5,q)$. The stabilizer of $\cQ(4,q)$ in $K$ certainly contains a group $H \simeq {\rm PSp}(2,q^2) \times 1 \le {\rm PSp}(2,q^2) \times {\rm PSp}(2n-2,q^2)$ acting in its natural representation on the $q^2+1$ lines of $\mu(\cD)$ contained in $\cQ(4,q)$. Since $N \in \langle \cQ(4,q) \rangle$, we have that $\langle \cQ(4,q) \rangle^\perp$ is contained in $\Sigma$. Also, since $H$ fixes $\langle \cQ(4,q) \rangle^\perp$ pointwise, it follows that $H$ stabilizes $\cQ^+(5,q)$. Under the Klein map, the action of $H$ on points of $\cQ^+(5,q)$ corresponds to the action of the stabilizer of an elliptic quadric $\cQ^-(3,q)$ in ${\rm PGL}(4,q)$ on lines of $\PG(3,q)$. Then $H$ has two orbits of size $q^2(q^2+1)/2$ on points of $\cQ^+(5,q) \setminus \cQ(4,q)$ corresponding to lines of $\PG(3,q)$ that are secant or external to $\cQ^-(3,q)$, respectively. Since $\pi$ and $\pi'$ share a line, they belong to different systems of generators of $\cQ^+(5,q)$. We may assume, that $\pi$ is a Latin plane and $\pi'$ is a Greek plane. Hence the $q^2$ points of $\pi \setminus \cQ(4,q)$ (resp. $\pi' \setminus \cQ(4,q)$) correspond to $q^2$ secant lines (resp. external lines) to $\cQ^-(3,q)$ and therefore they lie in the same $H$--orbit. Since $H \le K$, the line $\ell$ contains $q$ points of $\cO_1$. Analogously, the line $\ell'$ contains $q$ points of $\cO_2$.

If $P\not\in X$, then $\pi$ shares with $\cQ^-(4n+1,q)$ the $2q+1$ points of $\ell \cup \ell_P$. Let $L = \phi^{-1}(\mu^{-1}(\ell_P))$. Then $L \not\in \cQ^-(2n-1,q^2)$. If $r$ is an external line to $\cQ^-(2n-1,q^2)$ through $L$, then $\mu(\phi( r ))$ is a set of $q^2+1$ lines of $\mu(\cD)$ partitioning the point set of a $\cQ(4,q)$ contained in $\cQ$. In particular $\ell_P \in \mu(\phi( r ))$ and $\cQ(4,q)$ meets $\cQ^-(4n-1,q)$ in an elliptic quadric $\cQ^-(3,q)$. Indeed, since a line of $\mu(\cD)$ is either contained in $\cQ^-(4n-1,q)$ or it is tangent to it and $\cQ(4,q)$ contains $q^2+1$ lines of $\mu(\cD)$, it follows that $|\cQ(4,q)\cap\cQ^-(4n-1,q)|=q^2+1$. Let $\cH = \cQ^+(4n-5,q)$ be the  quadric $\langle \cQ^-(3,q) \rangle^{\perp_{|\Gamma}}$ and let $\cQ^-(5,q) = \langle \cH \rangle^\perp \cap \cQ^-(4n+1,q)$. By construction $\ell$ and $\ell'$ are lines of $\cQ^-(5,q)$ and $\cQ(4,q) \subset \cQ^-(5,q)$. The stabilizer of $\cQ(4,q)$ in $K$ certainly contains a group $H \simeq {\rm PSp}(2,q^2) \times 1 \le {\rm PSp}(2,q^2) \times {\rm PSp}(2n-2,q^2)$ acting in its natural representation on the $q^2+1$ lines of $\mu(\cD)$ contained in $\cQ(4,q)$. Since $N \in \langle \cQ(4,q) \rangle$, we have that $\langle \cQ(4,q) \rangle^\perp$ is contained in $\Sigma$. Also, since $H$ fixes $\langle \cQ(4,q) \rangle^\perp$ pointwise, it follows that $H$ stabilizes $\cQ^-(5,q)$. Under the Klein map, the action of $H$ on points of $\cQ^-(5,q)$ corresponds to the action of the stabilizer of an elliptic quadric $\cQ^-(3,q)$ contained in a symplectic polar space $\cW(3,q)$ embedded in a Hermitian surface $\cH(3,q^2)$ in ${\rm PGU}(4,q^2)$ on lines of $\cH(3,q^2)$. From \cite{PW}, $H$ has two orbits, $\cR_1$ and $\cR_2$, of size $q^2(q^2-1)/2$ on points of $\cQ^-(5,q) \setminus \cQ(4,q)$ and each of them is a relative hemisystem of $\cQ^-(5,q)$ with respect to $\cQ(4,q)$. In particular, we may assume that $P_i \in \cR_i$, $i = 1,2$ and, since $H \le K$, we have that $\cR_i \subset \cO_i$, $i = 1,2$. Also, from $\cite{PW}$, the group $H$ acts transitively on lines of $\cQ^-(5,q)$ that are not contained in $\cQ(4,q)$. Therefore $\ell$ and $\ell'$ lie in the same $H$--orbit and both contain $q/2$ points of $\cO_1$ and $q/2$ points of $\cO_2$. 

Let $\cL_i$ be the set of lines of $\cL$ containing $q$ points of $\cO_i$, $i = 1,2$ and $\cL_3$ be the set of lines of $\cL$ containing $q/2$ points of $\cO_1$ and $\cO_2$. Through the point $P_1$ (resp. $P_2$), there pass $(q^{2n-2}-1)(q^{2n}+1)/(q-1)$ lines of $\cL_1$ (resp. $\cL_2$) and $q^{2n-2}(q^{2n}+1)$ lines of $\cL_3$. From Proposition \ref{points}, the stabilizer of $P_1$ (resp. $P_2$) in $K$ is a group $J$ isomorphic to ${\rm PGO}^-(2n,q^2)$. The group $J$ has two orbits on points of $\cQ^-(4n-1,q)$, namely $X$ and its complement. Therefore $J$ (and hence $K$) permutes in a single orbit the lines of $\cL_i$ passing through $P_1$ (resp. $P_2$), $i = 1,2,3$. Hence, if $\ell_i \in \cL_i$, we have that 
$$|\ell_i^K| = \frac{q^{2n}(q^{2n}-1) (q^{2n-2}-1)(q^{2n}+1)}{2 q (q-1)}  = \frac{q^{2n-1}(q^{2n-2}-1)(q^{4n}-1)}{2 (q-1)}, i=1,2,$$ whereas, if $\ell_3 \in \cL_3$, we have that 
$$|\ell_3^K| = \frac{2 q^{2n}(q^{2n}-1) q^{2n-2}(q^{2n}+1)}{2 q}  = q^{4n-3}(q^{4n}-1).$$ On the other hand, $$\sum_{i=1}^{3} |\ell_i^K| = |\cL| = \frac{q^{2n-1}(q^{2n-1}-1)(q^{4n}-1)}{q-1}$$  and hence $\cL_i = \ell_i^{K}$ is an orbit under the action of $K$, $i = 1,2,3$.
\end{proof}

We are ready to prove the main result of this Section.

\begin{theorem}\label{relhem}
$\cO_i$ is a relative hemisystem of $\cQ^-(4n+1,q)$, $i=1,2$.
\end{theorem}
\begin{proof}
Let $g$ be a generator of $\cQ^-(4n+1,q)$ not contained in $\cQ$ and let $g'$ be the $(2n-2)$--space obtained by intersecting $g$ with $\cQ$. Let $P$ be a point of $g \setminus g'$ and let $\cQ^-(4n-1,q) = P^\perp \cap \cQ$. Then $g'$ is a generator of $\cQ^-(4n-1,q)$ and, as already observed, $\cQ^-(4n-1,q)$ contains $|\cQ^-(2n-1,q^2)|$ lines of $\mu(\cD)$. Let $X$ denote the set of points covered by the lines of $\mu(\cD)$ that are contained in $\cQ^-(4n-1,q)$. From \cite[Theorem 1.1, Table 1]{K}, the set $X$ is a $(q^{2n-2}-1)/(q-1)$--ovoid of $\cQ^-(4n-1,q)$, i.e., every generator of $\cQ^-(4n-1,q)$ meets $X$ in $(q^{2n-2}-1)/(q-1)$ points.   
If $P \in \cO_1$, then $g$ contains $(q^{2n-2}-1)/(q-1)$ lines of $\cL_1$ passing through $P$ and $q^{2n-2}$ lines of $\cL_3$ passing through $P$. It follows that $g$ contains 
$$\frac{(q-1)(q^{2n-2}-1)}{q-1} + \left(\frac{q}{2}-1\right) q^{2n-2} + 1 = \frac{q^{2n-1}}{2}$$ points of $\cO_1$ and 
$$\frac{q}{2} q^{2n-2} = \frac{q^{2n-1}}{2}$$ points of $\cO_2$. Analogously, if $P \in \cO_2$.
\end{proof}

\begin{remark}
{\rm The relative hemisystems of $\cQ^-(4n+1,q)$ constructed in Theorem \ref{relhem} admit ${\rm PSp}(2n,q^2)$ as an automorphism group.}
\end{remark}

\begin{remark}
{\rm With the aid of Magma \cite{magma} we checked that ${\cal Q}^-(7,2)$ has no relative hemisystems.}
\end{remark}

\section{Applications}\label{appl}

In this section we show that associated to a relative hemisystem of $\cQ^-(2n+1,q)$, $n \ge 2$, is a $\left( \frac{q^n-1}{q-1} \right)$--ovoid of $\cW(2n+1,q)$. Also, we provide a construction of a strongly regular graph with parameters $v=2^{n-1}(2^n-1)$, $k=(2^{n-2}-1)(2^n+1)$, $\lambda = 2(2^{n-1}+1)(2^{n-3}-1)$, $\mu = 2^{n-1}(2^{n-2}-1)$ from a relative hemisystem of $\cQ^-(2n+1,2)$, $n \ge 3$. Note that, when $n$ is even, this graph has the same parameters as the strongly regular graph constructed in \cite{HM}, by considering as vertices the elliptic sections of a parabolic quadric of $\PG(n,4)$.

\subsection{$\left( \frac{q^n-1}{q-1} \right)$--ovoids of $\cW(2n+1,q)$, $q$ even}

In $\cW(2n+1,q)$, $q$ even, the classical example of $\left( \frac{q^n-1}{q-1} \right)$--ovoid is an elliptic quadric $\cQ^-(2n+1,q)$ polarizing to the symplectic polarity of $\cW(2n+1,q)$. Our examples of $\left( \frac{q^n-1}{q-1} \right)$--ovoids are examples of {\em quasi--quadrics} in the sense that they have the same size and the same intersection numbers with respect to hyperplanes of a non--degenerate quadric of $\PG(2n+1,q)$.

Let $\cQ^-(2n+1,q)$, $n \ge 2$, be an elliptic quadric of $\PG(2n+1,q)$, let $\cQ=\cQ(2n,q)$ be a parabolic quadric embedded in $\cQ^-(2n+1,q)$ and let $\cR$ be a relative hemisystem of $\cQ^-(2n+1,q)$ with respect to $\cQ$. Since $q$ is even, the polarity induced by $\cQ^-(2n+1,q)$ gives rise to a symplectic polar space $\cW(2n+1,q)$. If $\delta$ is a non--zero element of $\GF(q)$ such that $X^2+X+\delta = 0$ is irreducible over $\GF(q)$, then we may assume that $\cQ^-(2n+1,q)$ is given by 
$$
X_1 X_{2n+2} + X_2 X_{2n+1} + \ldots + X_{n} X_{n+3} + X_{n+1}^2 + X_{n+1} X _{n+2} + \delta X_{n+2}^2 = 0 , 
$$
whereas $\cQ$ is given by
$$
X_1 X_{2n+2} + X_2 X_{2n+1} + \ldots + X_{n} X_{n+3} + X_{n+1}^2 = 0 . 
$$
Let $\cQ^+(2n+1,q)$ be the hyperbolic quadric given by
$$
X_1 X_{2n+2} + X_2 X_{2n+1} + \ldots + X_{n} X_{n+3} + X_{n+1}^2 + X_{n+1} X _{n+2} = 0 . 
$$
Then the quadrics $\cQ^-(2n+1,q)$ and $\cQ^+(2n+1,q)$ share $\cQ(2n,q)$ and generate a pencil $\cP$ comprising $q/2$ elliptic quadrics, $q/2$ hyperbolic quadrics and a cone $\cC$ having as basis the parabolic quadric $\cQ(2n,q)$ and as vertex its nucleus $N$. Since $q$ is even, an orthogonal polarity gives rise to a symplectic polarity. In particular, the orthogonal polarities induced by the non--degenerate quadrics of $\cP$ in $\PG(2n+1,q)$ all polarize to the same symplectic polarity $\perp$ of $\PG(2n+1,q)$ which defines the symplectic polar space $\cW(2n+1,q)$.

Each elliptic quadric $\cQ^-_i$ of $\cP$ is partitioned as follows: $\cQ^-_i=\cR_i \cup \cR_i^\tau \cup \cQ$, $i=1 \dots q/2$, where with $\cR_i$ and $\cR_i^\tau$ we have denoted the points on $\cQ^-_i$ corresponding to $\cR$ and its complement in $\cQ^-_i \setminus \cQ$, respectively.
	
Let $\cQ^-_i$ and $\cQ^-_j$ be two elliptic quadrics of $\cP$. Let $\cA$ be one of the two subsets $\cR_i$ and $\cR_i^c$ on $\cQ^-_i$ and let $\cB$ be one of the two subsets $\cR_j$ and $\cR_j^c$ on $\cQ^-_j$, $i \ne j$. We show that $\cX = \cA \cup \cB \cup \cQ$ is a $(q^n-1)/(q-1)$--ovoid of $\cW(2n+1,q)$. Let $\pi$ be a totally isotropic $n$--space of $\cW(2n+1,q)$. Three possibilities occur:

\begin{itemize}

\item[1.] $\pi$ meets $\cQ$ in an $(n-1)$--space and $\pi$ is contained in $N^\perp$. There are $\prod_{i=1}^n (q^i+1)$ such $n$--spaces and none of these meets $\cX$ furtherly.

\item[2.] $\pi$ meets $\cQ$ in an $(n-1)$--space and $\pi$ is not contained in $N^\perp$. Then $\pi$ will be contained in a hyperbolic quadric of $\cP$ and $\pi$ does not meet $\cX$ furtherly.

\item[3.] $\pi$ meets $\cQ$ in an $(n-2)$--space $\sigma$. Then $\pi \cap \cQ^-_i$ is an $(n-1)$--space $\sigma_i$ and $\pi \cap \cQ^-_j$ is an $(n-1)$--space $\sigma_j$, $i\ne j$. In this case $\sigma_i \cap \sigma_j = \sigma$. Since $\cA$ and $\cB$ are relative hemisystems of $\cQ_i$ and $\cQ_j$, respectively, we have that $\sigma_i$ meets $\cA$ in $q^{n-1}/2$ points and $\sigma_j$ meets $\cB$ in $q^{n-1}/2$ points. Then, $\pi$ meets $\cX$ in $(q^n-1)/(q-1)$ points.
\end{itemize}
We have proved the following proposition.

	\begin{prop}
A relative hemisystem of $\cQ^-(2n+1,q)$ gives rise to $2^{q/2}$ non--classical $(q^n-1)/(q-1)$--ovoids of $\cW(2n+1,q)$.	
	\end{prop}

	\begin{cor}
The set $\cX$ is an elliptic quasi--quadric.	
	\end{cor}
	\begin{proof}
The set $\cX$ has the same size of an elliptic quadric $\cQ^-(2n+1,q)$ of $\PG(2n+1,q)$ and from \cite[Theorem 11]{BKLP} is a two--character set with intersection numbers $h_1 = (q^n-1)(q^n+1)/(q-1)$ and $h_2 = (q^n-1)(q^n+1)/(q-1) - q^n$.
	\end{proof}
	
The above construction of non-classical $(q^n-1)/(q-1)$-ovoids of $\cW(2n+1,q)$ could be realized by taking on the two elliptic quadrics $\cQ^-_i$ and $\cQ^-_j$ two subsets $\cA$ and $\cB$ corresponding to two non--isomorphic relative hemisystems.

	\subsection{A strongly regular graph}
	
Let $\cal R$ be a relative hemisystem of $\cQ^-(2n+1,2)$, $n\ge 3$.  Here $\perp$ denotes the orthogonal polarity of $\PG(2n+1,2)$ induced by $\cQ^-(2n+1,2)$. Let $\cal G$ be the graph having as vertices the points of $\cal R$ with two vertices $x,y$ being adjacent if $xy$ is a line of  $\cQ^-(2n+1,2)$. 

	\begin{prop}
The graph $\cal G$ is strongly regular with parameters $v=2^{n-1}(2^n-1)$, $k=(2^{n-2}-1)(2^n+1)$, $\lambda = 2(2^{n-1}+1)(2^{n-3}-1)$, $\mu = 2^{n-1}(2^{n-2}-1)$.	
	\end{prop}
	\begin{proof}
The number of vertices is $v=2^{n-1}(2^n-1)$ and from Lemma \ref{pre} $d)$ we have that $k=(2^{n-2}-1)(2^n+1)$.
Let $x,y\in\cal R$ such that $x$ is not adjacent to $y$. Then $y\in (x^\tau)^\perp\setminus\{x^\tau\}$. We have that $\mu=\vert x^\perp \cap y^\perp\cap\cal R\vert$ and $x^\perp\cap y^\perp$ is an elliptic quadric $\cQ^-(2n-1,2)$ meeting $\cal Q$ in the cone $P\cQ^-(2n-3,2)$, where $P=x^\tau y\cap {\cal Q}$. There are three hyperplanes through $x^\perp\cap y^\perp$, namely, $x^\perp$, $y^\perp$ and $T^\perp$, where $T=\Sigma\cap xy$ and $\Sigma=\langle{\cal Q}\rangle$. Then $T^\perp$ meets $\cal Q$ in a cone $P\cQ(2n-2,2)$ containing $P\cQ^-(2n-3,2)$. Hence, from Lemma \ref{pre} $c),d)$, $\vert T^\perp\cap{\cal R}\vert =2^{2n-2}$ and $\vert x^\perp\cap\cR\vert =\vert y^\perp\cap\cR\vert=2^{n-2}(2^n+1)-2^n$. It follows that 
$$ \mu+2(2^{n-2}(2^n+1)-2^n-\mu)+2^{2n-2}-\mu=2^{n-1}(2^n-1)$$ and hence $\mu=2^{n-1}(2^{n-2}-1)$.
Let $x,y\in\cal R$ such that $x$ is adjacent to $y$. Then $y\in x^\perp\cap\cR\setminus\{x\}$. We have that $\lambda=\vert x^\perp \cap y^\perp\cap\cR\setminus\{x,y\}\vert$ and $x^\perp\cap y^\perp$ is a cone $\ell\cQ^-(2n-3,2)$, where $\ell=xy$, meeting $\cal Q$ in the cone $T\cQ^-(2n-3,2)$, with $T=\ell\cap\cQ$.
There are three hyperplanes through $x^\perp\cap y^\perp$, namely, $x^\perp$, $y^\perp$ and $T^\perp$. Hence, from Lemma \ref{pre} $d),f)$, $\vert T^\perp\cap{\cal R}\vert =2^{n-1}(2^{n-1}-1)$ and $\vert x^\perp\cap\cR\vert =\vert y^\perp\cap\cR\vert=2^{n-2}(2^n+1)-2^n$. It follows that 
$$
\lambda+2+2(2^{n-2}(2^n+1)-2^n-\lambda-2)+(2^{n-1}(2^{n-1}-1)-\lambda-2)=2^{n-1}(2^n-1)$$ and hence $\lambda=2^{2n-3}-2^n+2^{n-2}-2 = 2(2^{n-1}+1)(2^{n-3}-1)$. 
	\end{proof}
	
	\begin{remark}
{\rm The graph $\cal G$ arising from the relative hemisystem of $\cQ^-(9,2)$ constructed in Theorem \ref{relhem} has the same parameters of the graph constructed in \cite{HM} by considering the elliptic sections of a parabolic quadric of $\PG(4,4)$.
With the aid of Magma \cite{magma} we checked that the two graphs have the same automorphism group isomorphic to $2\cdot {\rm PSp}(4,4)$.
Hence it is plausible that the two graphs are isomorphic.}
	\end{remark}

\end{document}